\documentclass[leqno,11pt]{amsart}

\usepackage{setspace}
\usepackage[english]{babel}
\usepackage{bussproofs}
\usepackage[utf8]{inputenc}
\usepackage{csquotes}
\usepackage{mathtools}
\usepackage{mathrsfs}
\usepackage{latexsym}
\usepackage{enumitem}
\usepackage{amsthm}
\usepackage{amssymb}
\DeclareMathAlphabet{\mathpzc}{OT1}{pzc}{m}{it}

\usepackage{bbm}

\newtheorem{theorem}{Theorem}[section]
\newtheorem*{theorem*}{Theorem}

\newtheorem{conjecture}{Conjecture}[section]
\newtheorem{lemma}[theorem]{Lemma}
\newtheorem{proposition}[theorem]{Proposition}
\newtheorem{corollary}[theorem]{Corollary}
\newtheorem{observation}[theorem]{Observation}
\newtheorem{fact}[theorem]{Fact}
\newtheorem{claim}[theorem]{Claim}
\theoremstyle{definition}
\newtheorem{definition}[theorem]{Definition}

\theoremstyle{remark}
\newtheorem{remark}{Remark}

\newtheorem{question}{Question}

\def\hook{\upharpoonright}
\def\forces{\Vdash}

\def\Me{\mathcal M}
\def\Null{\mathcal N}
\def\ZFC{\mathsf{ZFC}}

\def\PFA{\mathsf{PFA}}
\def\MA{\mathsf{MA}}

\def\BA{\mathsf{BA}}

\def\baire{\omega^\omega}

\def\mfb{\mathfrak b}
\def \mfd{\mathfrak{d}}
\def\GCH {\mathsf{GCH}}
\def\CH {\mathsf{CH}}
\def\Q{\mathbb Q}
\def\P{\mathbb P}

\def\cc{2^{\aleph_0}}

\def\mfp{\mathfrak{p}}

\def\non{{\rm non}}
\def\cov{{\rm cov}}

\title{Weak Baumgartner Axioms and Universal Spaces}
\author[Switzer]{Corey Bacal Switzer}
\address[C.~B.~Switzer]{Institut f\"{u}r Mathematik, Kurt G\"odel Research Center, Universit\"{a}t Wien, Kolingasse 14-16, 1090 Wien, AUSTRIA}
\email{corey.bacal.switzer@univie.ac.at}

\thanks{\emph{Acknowledgments:} This research was funded in whole or in part by the Austrian Science Fund (FWF) through the following grant: 10.55776/ESP548.}

\begin{document}

\begin{abstract}
    If $X$ is a topological space and $\kappa$ is a cardinal then $\BA_\kappa (X)$ is the statement that for each pair $A, B \subseteq X$ of $\kappa$-dense subsets there is an autohomeomorphism $h:X \to X$ mapping $A$ to $B$. In particular $\BA_{\aleph_1} (\mathbb R)$ is equivalent the celebrated Baumgartner axiom on isomorphism types of $\aleph_1$-dense linear orders. In this paper we consider two natural weakenings of $\BA_\kappa (X)$ which we call $\BA^-_\kappa (X)$ and $\mathsf{U}_\kappa (X)$ for arbitrary perfect Polish spaces $X$. We show that the first of these, though properly weaker, entails many of the more striking consequences of $\BA_\kappa (X)$ while the second does not. Nevertheless the second is still independent of $\ZFC$ and we show in particular that it fails in the Cohen and random models. This motivates several new classes of pairs of spaces which are ``very far from being homeomorphic" which we call ``avoiding", ``strongly avoiding", and ``totally avoiding". The paper concludes by studying these classes, particularly in the context of forcing theory, in an attempt to gauge how different weak Baumgartner axioms may be separated.
\end{abstract}

\maketitle

\section{Introduction}

Cantor's celebrated theorem on countable dense linear orders states (in part) that given any two countable dense $A, B \subseteq \mathbb R$ there is a linear order isomorphism $h:A \to B$ between them. It's easy to see that this in turn immediately implies that $\mathbb R$ is {\em CDH} (``countable dense homogeneous") i.e. that for $A$ and $B$ countable and dense sets of reals there is an autohomeomorphism $h:\mathbb R \to \mathbb R$ so that $h``A = B$. It is natural to ask whether Cantor's theorem can lift to the uncountable. Recall that a set $A \subseteq \mathbb R$ is called {\em $\aleph_1$-dense} if its intersection with each nonempty interval has size $\aleph_1$. Baumgartner in \cite{Baum73} proved the consistency of ``all $\aleph_1$-dense sets of reals are order isomorphic", which has since become known as $\BA$ for ``Baumgartner's axiom". Baumgartner's axiom has a natural topological phrasing which allows one to state parametrized principles $\BA (X)$ for arbitrary topological spaces $X$ ($\BA$ itself is $\BA (\mathbb R)$), see below for precise definitions. In this paper we study these parametrized principles in the general setting of Polish spaces i.e. separable completely metrizable spaces with an eye towards comparing their related Baumgartner axioms.

Already it was known from the work of Abraham and Shelah \cite{AvrahamShelah81}, Stepr\={a}ns-Watson \cite{Stepranswatson87}, Baldwin-Beaudoin \cite{BaldwinBeaudoin89} and others that there are essential differences between the case of the real line $\mathbb R$ and the cases of non one-dimensional spaces like $\mathbb R^n$ with $n > 1$ or the Cantor and Baire spaces $2^\omega$ and $\omega^\omega$. For example, while the latter three have their associated Baumgartner axioms follow from $\MA$ (and indeed even $\mathfrak{p} > \aleph_1$), $\BA$ itself does not follow from $\MA$. Very strikingly Todor\v{c}evi\'c and Guzm\'an recently showed that this remains true even in the presence of the $P$-ideal dichotomy (\cite{TodorcevicGuzman24}). It follows from these results that in particular $\BA = \BA (\mathbb R)$ is not implied by its higher dimensional analogues, $\BA (\mathbb R^n)$. Left open is whether the converse is true, which was conjectured to be the case by Stepr\={a}ns and Watson in \cite{Stepranswatson87}.

\begin{conjecture}(Stepr\={a}ns-Watson, \cite{Stepranswatson87})
    $\BA$ implies $\BA (\mathbb R^n)$ for all finite $n > 1$. \label{conj}
\end{conjecture}
The Stepr\={a}ns-Watson conjecture is still open and motivates in some sense this work. In \cite{Stepranswatson87} Stepr\={a}ns and Watson observed that to confirm Conjecture \ref{conj} it would suffice to show that $\BA$ implies $\mathfrak{p} > \aleph_1$, noting that Todor\v{c}evi\'c had already shown $\BA$ implies $\mathfrak b > \aleph_1$ (\cite{todorcevic89}), thus making the possibility that $\BA$ affects other cardinal characteristics plausible. 

\begin{question}
    Does $\BA$ imply $\mathfrak{p} > \aleph_1$?
\end{question}
This question is also open and has been mentioned throughout the literature, for example see the introduction to \cite{TodorcevicGuzman24}. 

In attempting to wrestle with these ideas we were led to a more general study of ``topological Baumgartner theory" which lifts the study of CDH spaces to the uncountable. In conversations with Andrea Medini relating to this he suggested natural weakenings of $\BA$ and asked whether they could be separated. The main results of this paper show how to separate these weakenings as well as place them more generally in the context of other more well studied axioms and models. 

In order to explain the content of the current paper now more clearly we recall the following definitions. Given a cardinal $\kappa$ and a topological space $X$, a set $A \subseteq X$ is said to be $\kappa$-{\em dense} if $A \cap U$ has size $\kappa$ for every nonempty open subset $U \subseteq X$. We denote by $\BA_\kappa (X)$ the statement that for every pair $A, B \subseteq X$ which are $\kappa$-dense there is an autohomeomorphism $h:X \to X$ so that $h``A = B$. When $\kappa = \aleph_1$ we often drop the subscript. It's not hard to check that $\BA$ is equivalent to $\BA (\mathbb R)$, see the Proposition on page 11 of \cite{Stepranswatson87}. An obvious weakening of $\BA_\kappa (X)$, for any $\kappa$ and $X$, is simply to say that all $\kappa$-dense subsets of $X$ are homeomorphic. We denote this statement\footnote{A more finegrained approach to these axioms is given in Section 2 below.} by $\BA^-_\kappa (X)$. Unlike the standard Baumgartner axioms, $\BA_\kappa^-(X)$ is equivalent for $X$ e.g. Baire space, Cantor space or the real line and in fact, for these it is equivalent to ``for all perfect Polish spaces $X$, $\BA^-_\kappa (X)$ holds", see Lemma \ref{BA-equivs} below. Hence, at least in the case of the real line, this is a proper weakening. Nevertheless, it implies some of the more striking consequences of $\BA$. 
\begin{theorem}
    Let $X$ be a perfect Polish space. The axiom $\BA_{\aleph_1}^-(X)$ implies the following: 
    
    \begin{enumerate}
        \item (Medini, Unpublished) $\mfb > \aleph_1$

    \item (Lemma \ref{BA-impliesCA} below)  $2^{\aleph_1} = 2^{\aleph_0}$
    \end{enumerate}
\end{theorem}

An even further weakening is to say that there is a $\kappa$-dense subset $Z \subseteq X$ into which every other $A \subseteq X$ of size $\kappa$ can be embedded. We denote this statement by $\mathsf{U}_\kappa (X)$ ($\mathsf{U}$ for ``universal"). While we have little to say currently about the $\BA^-_{\kappa}(X)$ axioms, there is in fact much that can be said about these axioms and unwinding them takes up the majority of the paper. 

The $\mathsf{U}_\kappa (X)$ axioms are clearly very weak and hold in $\ZFC$ for $\kappa = \aleph_0$ or $\kappa = 2^{\aleph_0}$ for, e.g. $X = 2^\omega$ (in fact the parametrizations are superfluous as we will see). Nevertheless for intermediate values of $\kappa$ we show that they may fail and in fact does so in the Cohen and random models.

\begin{theorem}[See Theorem \ref{Ufailsthm} below]
    If $\kappa < \lambda$ are uncountable cardinals then $\mathsf{U}_\kappa (X)$ fails in the model obtained by adding either $\lambda$ many Cohen reals or $\lambda$ many random reals for any perfect Polish space $X$. \label{Cohen1}
\end{theorem}

However $\mathsf{U}_\kappa (X)$ fails to have either of the consequences listed above.
\begin{theorem}
    It is consistent that $2^{\aleph_0}=\aleph_2$ and $\mathsf{U}_{\aleph_1} (\mathbb R)$ holds but $2^{\aleph_1} = \aleph_3$ and $\mfb = \aleph_1$.  \label{noconsequencesofU1}
\end{theorem}

A corollary of this is the following.
\begin{corollary}
    $\mathsf{U}_{\aleph_1}(\mathbb R)$ does not imply $\BA^-_{\aleph_1}(\mathbb R)$. 
\end{corollary}

The proof of Theorem \ref{Cohen1} gives significantly more information and suggests a class of pairs of spaces which are ``very far" from being homeomorphic. We finish the paper by studying three variations of this class which we dub ``avoiding", ``strongly avoiding", and ``totally avoiding". We show that all three of these notions are preserved by finite support iterations of ccc forcing notions. This allows us to prove a general ``no step up theorem". A simple version of this is the following.
\begin{theorem}
    It is consistent that $\BA$ holds, $2^{\aleph_0} = \aleph_3$ and $\mathsf{U}_{\aleph_2} (\mathbb R)$ fails. In fact in this model if $Z \subseteq \mathbb R$ has size $\aleph_2$ then there is some $Y_Z \subseteq \mathbb R$ also of size $\aleph_2$ so that every continuous $f:Y_Z \to Z$ has countable range. \label{nostepupintro}
\end{theorem}

In the statement of the theorem above the property described of $Z$ and $Y_Z$ is ``$Y_Z$ totally avoids $Z$". A more general statement can be found as Theorem \ref{nostepup} below. 

The rest of this paper is organized as follows. In the remainder of this section we recall some preliminaries we will need. The following section defines the weakenings we will study and makes some easy observations. In Section 3 we look at the first of these weakenings and show that while it is indeed different from $\BA$ it has many of the same applications. Meanwhile Section 4 shows that an even further weakening fails to have any of these applications but still may fail badly. Finally in Section 5 we exploit this bad failure to discuss the possible behavior of homeomorphism types of subsets of Polish spaces via the ``avoiding" notions described above and prove the above advertised iteration theorem for ccc forcing notions in this context which we then put to use in deriving Theorem \ref{nostepupintro} above. We also make some general observations about these ideas. The paper then concludes with a list of open problems.
\medskip

\noindent {\bf Acknowledgements.} I would like to thank Andrea Medini, Juris Stepr\=ans and Lyubomyr Zdomskyy for several extremely helpful and enlightening conversations about the material contained in this paper. I would particularly like to thank Andrea Medini for also allowing me to include Theorem \ref{medini}.

\subsection{Preliminaries}

Unless otherwise stated all topological spaces we consider in this paper are separable and metrizable. All of our set theoretic notation is standard and the reader is referred to the monographs \cite{KenST} and \cite{JechST} for more. For descriptive set theory we have followed \cite{Kechrisbook}, which we refer the reader to for more background on this area. For details about cardinal characteristics we also suggest \cite{BarJu95} and the handbook article \cite{BlassHB}. We only remark that if $A$ and $B$ are sets, $C \subseteq A$ and $h:A \to B$ is a function then $h``C$ denotes the set $\{h(c) \; | \; c \in C\}$, which is sometimes in the literature also written as $h[C]$.

A {\em Polish space} is a topological space whose topology is compatible with a separable, complete metric. Such a space is {\em perfect} if it has no isolated points. An important fact we will use often is that if $P$ is a perfect Polish space then every non-empty open $U \subseteq P$ contains a (closed) copy of the Cantor space $2^\omega$. Spaces with this latter property are referred to as {\em Cantor crowded}. We need the following facts, due to Kuratowski and Lavrentiev respectively.

\begin{fact} \label{Lavthm}
    Let $X$ and $Y$ be Polish spaces and $A \subseteq X$. Let $f:A \to Y$ be continuous. 
    \begin{enumerate}
        \item (Kuratowski, see Theorem 3.8 of \cite{Kechrisbook}) There is a $G_\delta$ subset $W \subseteq X$ so that $A \subseteq W$ and a continuous function $\hat{f}: W \to Y$ so that $\hat{f} \supseteq f$.
        \item (Lavrentiev, see Theorem 3.9 of \cite{Kechrisbook}) If $f$ is a homeomorphism, then there are $G \subseteq X$ and $H \subseteq Y$ which are $G_\delta$ subsets of $X$ and $Y$ respectively so that $A \subseteq G$ and ${\rm im}(f) \subseteq H$ and $\hat{f}:G \to H$ is a homeomorphism extending $f$. In particular if $f$ is a homeomorphism onto its image then $\hat{f}$ as in part (1) will be as well.
    \end{enumerate}
    
\end{fact}

A proof of both parts can be found in \cite[p. 16]{Kechrisbook}. If we are considering a continuous function $f$ from some subset of a Polish space into another Polish space then any $\hat{f}$ from Fact \ref{Lavthm} is called a {\em lift} of $f$. Sometimes in a lazy abuse of notation we refer simply to $\hat{f}$ as {\em the lift}, though it may not be unique. In such cases the point is that any lift will work for whatever argument is at hand and we do not worry to distinguish between them since any two lifts agree on their common domain, which will again be a $G_\delta$, if e.g. $A$ is dense in $X$.

We will also be often referring to Polish spaces (and their Borel subsets) in the context of forcing. There are many (standard) ways to treat these, here is one for concreteness. Recall that a space if Polish if and only if it is a $G_\delta$ subset of the Hilbert cube, $\mathbb I^\omega$. If $M$ is a model of set theory then we say that a Polish space $X$ is {\em coded in} $M$ if a Borel code for a $G_\delta$ subset of $\mathbb I^\omega$ homeomorphic to $X$ is in $M$. In this case we can reinterpret $X$ in a forcing extension in the natural way. 

Finally we recall the definitions of some cardinal characteristics that will be discussed. Denote by $\Me$ the meager ideal. The cardinal $\non (\Me)$ is the least size of a non meager set and the cardinal $\cov (\Me)$ is the least size of a family of meager sets needed to cover the real line. If $f, g \in \baire$ then we write $f \leq^* g$ if for all but finitely many $k \in \omega$ we have that $f(k) \leq g(k)$. The bounding number $\mfb$ denotes the least size of a family $\mathcal F \subseteq \baire$ which is unbounded i.e. so that there is no single $f \in \baire$ with $g \leq^* f$ for all $g \in \mathcal F$ simultaneously. Dually the dominating number $\mfd$ is the least size of a dominating family - that is a family $\mathcal D \subseteq \baire$ so that every $g \in \baire$ is $\leq^*$-below some $f \in \mathcal D$. Finally recall that if $A, B \subseteq \omega$ are infinite then we write $A \subseteq^* B$ if there is a $k \in \omega$ with $A \setminus k \subseteq B$. If $\mathcal F$ is a family of infinite subsets of natural numbers we say that $\mathcal F$ has the {\em strong finite intersection property} if there for every finitely subset $\mathcal A \subseteq \mathcal F$ we have that $\bigcap \mathcal A$ is infinite. If $A \subseteq \omega$ is infinite then we say that $A$ is a {\em pseudointersection} of such a family $\mathcal F$ if $A \subseteq^* B$ for all $B \in \mathcal F$. The pseudointersection number $\mfp$ is the least size of a family with the strong finite intersection property but no pseudointersection. For the reader's convienience we recall the following well known facts, see e.g. \cite{BlassHB}.

\begin{fact}
    The follow inequalities are provable in $\ZFC$. In each case they many be strict.

    \begin{enumerate}
        \item $\mfp \leq \mfb \leq \mfd$
        \item $\mfb \leq \non (\Me)$ and $\cov (\Me) \leq \mfd$
        \item $\mfp \leq \cov(\Me)$
    \end{enumerate}
\end{fact}
 
\section{Some Weakenings of $\BA$}

In this section we introduce two families of weakenings of $\BA_\kappa (X)$ and make some initial observations. The initial axioms were suggested by Andrea Medini and their parametrized versions appeared in the proofs given in sections 3-5 of this paper in an effort to answer basic questions about them. First we recall a standard definition from topology.

\begin{definition}
    Let $X$ be a (non-empty) topological space and $\kappa$ an infinite cardinal. We say that $X$ is $\kappa$-{\em crowded} if every non-empty open set has size $\kappa$.
\end{definition}

For example any perfect Polish space is $2^{\aleph_0}$-crowded. 

Recall that a topological space is {\em zero dimensional} if it has a basis of clopen sets. We will use the following two facts about zero dimensional spaces frequently.
\begin{fact}[See Theorem 7.8 of \cite{Kechrisbook} and its proof]
    Let $Y$ be zero-dimensional\footnote{Again, throughout we are assuming that $Y$ such as this are separable and metrizable. We won't remark on this again but many of the statements, including this one, rely on this assumption.} and assume $\kappa \leq 2^{\aleph_0}$. If $Y$ is $\kappa$-crowded then $Y$ is homeomorphic to a $\kappa$-dense subset of $2^\omega$. If $\kappa < \cc$ or $Y$ is nowhere locally compact then $Y$ is also homeomorphic to a $\kappa$-dense subset of $\omega^\omega$.\label{embedintocantorspace}
\end{fact}

\begin{fact}(See \cite[Corollary 6.2.8]{Engleking})
    If $\kappa < 2^{\aleph_0}$ and $Z$ is of size $\kappa$ then $Z$ is zero dimensional.
\end{fact} 


Here are the axioms we will be considering in this paper.
\begin{definition}[Medini]
    Let $\kappa \leq 2^{\aleph_0}$ a cardinal. 

    \begin{enumerate}
        \item The axiom $\BA^-_\kappa$ states that every pair of zero-dimensional, $\kappa$-crowded spaces are homeomorphic. Denote by $\BA^-$ the statement $\BA^-_{\aleph_1}$.

        \item The axiom $\mathsf{U}_\kappa$ states that there is a zero dimensional space $Z$ into which any other zero dimensional space of size $\kappa$ homeomorphically embeds.
    \end{enumerate}
\end{definition}

In the definition of $\mathsf{U}_\kappa$ we call a $Z$ as described above ``universal" (for $\kappa$-sized spaces). To be clear, we mean by $Y$ embeds into $Z$ that there is a continuous injection $f:Y \to Z$ so that the image of $f$ is  homeomorphic to $Y$. Clearly for any $\kappa$ we have that $\BA^-_\kappa$ implies $\mathsf{U}_\kappa$ since the unique (up to homeomorphism) $\kappa$-crowded zero dimensional, (separable, metrizable) space is universal. 

We will also consider parametrized versions of these principles.

\begin{definition}
    Let $X$ be a perfect Polish space and $\kappa \leq \lambda \leq 2^{\aleph_0}$ a cardinal. 

    \begin{enumerate}
        \item The axiom $\BA^-_\kappa (X)$ states that all $\kappa$-dense subsets of $X$ are homeomorphic. Denote by $\BA^- (X)$ the statement $\BA^-_{\aleph_1} (X)$.

        \item The axiom $\mathsf{U}_{\kappa, \lambda} (X)$ states that there is a subset $Z \subseteq X$ which has size $\lambda$ so that for every $W \subseteq X$ of size $\kappa$ there is an embedding $f:W \to Z$. Denote by $\mathsf{U}_{\kappa}(X)$ the statement $\mathsf{U}_{\kappa, \kappa} (X)$. 
    \end{enumerate}
\end{definition}

We note that in the definition $\mathsf{U}_{\kappa, \lambda} (X)$ again that ``embedding" means homeomorphism of its image. In what follows we could have weakened the requirement in $\mathsf{U}_\kappa$ that the function $f$ be a homeomorphic embedding to simply injection as frequently in showing the failure of $\mathsf{U}_\kappa$ in certain models we often show the stronger statement that for every $Z$ there is some $Y$ which does not continuous inject into $Z$. This was simply a choice of taste. As in the non parametrized versions it's clear that for any $\kappa$ and $X$ we have that $\BA^-_\kappa(X)$ implies $\mathsf{U}_{\kappa} (X)$ and $\mathsf{U}_{\kappa, \lambda}(X)$ implies $\mathsf{U}_{\kappa, \lambda '}(X)$ whenever $\lambda < \lambda '$. 

It turns out in the case of $\BA^-$ the parametrizations are superfluous.

\begin{lemma}
    Let $\kappa$ be a cardinal. The following are equivalent:
    \begin{enumerate}
        \item $\BA^-_\kappa$
        \item $\BA^-_\kappa (2^\omega)$
        \item $\BA^-_\kappa(\baire)$
        \item $\BA^-_\kappa (X)$ holds for some perfect Polish space $X$
        \item $\BA^-_\kappa (X)$ holds for every perfect Polish space $X$. 
    \end{enumerate}
    \label{BA-equivs}
\end{lemma}

\begin{proof}
    All of the above conditions fail if $\kappa = 2^{\aleph_0}$ since every perfect Polish space has $2^{\aleph_0}$-dense sets which are not homeomorphic. This is essentially well known but we include a proof for completeness. Let $X$ be a perfect Polish space. Since $X$ itself is $2^{\aleph_0}$-dense in itself we just have to show that it cannot be homeomorphic to every one of its $2^{\aleph_0}$-dense subsets. Let $\{U_n\}_{n < \omega}$ be a basis for $X$ and let $C_n \subseteq U_n$ be a closed, nowhere dense set of size continuum for each $n < \omega$. Let $C = \bigcup_{n < \omega} C_n$. This is a meager (in the sense of $X$), $2^{\aleph_0}$-dense subset of $X$. Suppose $h: X \to C$ were a homeomorphism. By Fact \ref{Lavthm}, $h$ lifts to a homeomorphism between $G_\delta$-subsets $X' \supseteq X$ and $C' \supseteq C$ but since $X$ has no supersubset in itself there is no proper lift so $C$ must be a dense $G_\delta$ in fact, contradicting the Baire category theorem\footnote{This fact can also be proved by a counting argument using Lemma \ref{BA-impliesCA} below. That lemma states (in part) that $\BA^-_\kappa$ implies $2^\kappa = \cc$. However $2^{\cc}$ of course can never be equal to $\cc$.}. 

    Therefore we can assume that $\kappa < 2^{\aleph_0}$ and hence that all $\kappa$-crowded spaces we will consider are in fact zero dimensional. The equivalence now of (1), (2) and (3) follows from Fact \ref{embedintocantorspace}. The implications (5) implies (4), (5) implies (3) and (5) implies (2) are immediate. Therefore it suffices to show that (1) implies (5) and (4) implies (3). Let's start with (1) implies (5). This is by contrapositive - suppose that $X$ is a perfect Polish space and $A, B \subseteq X$ are $\kappa$-dense but not homeomorphic. Since $\kappa < 2^{\aleph_0}$ the spaces $A$ and $B$ are zero dimensional, separable, $\kappa$-crowded spaces which are not homeomorphic hence (1) fails. 

    Finally let us show that (4) implies (3). Assume (4) and let $X$ be a perfect Polish space so that $\BA_\kappa^-(X)$ holds. Since $\baire$ can be written as $2^\omega \setminus D$ for a countable dense set $D$, and $2^\omega$ can be embedded into any open subset of $X$ as a closed, nowhere dense set, for any nonempty open $U \subseteq X$ there is a closed nowhere dense $C \subseteq U$ with $\baire$ homeomorphic to a subset of it. Fix $A$ and $B$ which are $\kappa$-dense subsets of $\baire$. We need to show that they are homeomorphic to $\kappa$-dense subsets of $X$ and hence, by $\BA^-_\kappa (X)$, each other. We will just show that $A$ can be embedded into $X$ as a $\kappa$-dense subset, the argument for $B$ is of course symmetric. For each $n < \omega$ let $A_n$ be the intersection of $A$ with the basic open of all sequences in $\baire$ starting in $n$. Note that $A_n$ is $\kappa$-dense in this open set, which is homeomorphic to $\baire$, and $A$ is the (topological) disjoint union of the $A_n$'s. Let $\{U_n\}_{n < \omega}$ be a base for the topology on $X$. Let $C_0$ be a closed, nowhere dense copy of $2^\omega$ in $U_0$. Now inductively for $n < \omega$ let $C_{n+1} \subseteq U_n \setminus \bigcup_{j < n+1} C_n$ be a closed, nowhere dense copy of $2^\omega$. For each $n < \omega$ now let $f_n:A_n \to C_n$ map $A_n$ homeomorphically into $C_n$. Then $\bigcup_{n < \omega} f_n``A_n$ is a $\kappa$-dense copy of $A$ in $X$ as needed. 
\end{proof}


Let us single out that in particular the following holds.
\begin{proposition}
    For every uncountable cardinal $\kappa$ the axioms $\BA^-_\kappa$ and  $\BA^-_\kappa (\mathbb R)$ are equivalent.
\end{proposition}


Lemma \ref{BA-equivs} gives several immediate separations. First we note the following.

\begin{lemma}
    If $\mathfrak{p} > \kappa$ then $\BA^-_\kappa$ holds. In particular for every perfect Polish space $X$ it is consistent that $\BA^-_\kappa (X)$ holds.
\end{lemma}

This is simply by the main result of \cite{BaldwinBeaudoin89}, see also \cite[Theorem 2.1]{Medini15}, that $\mfp > \kappa$ implies $\BA_\kappa (2^\omega)$ and hence $\BA^-_\kappa(2^\omega)$. The last point is worth noting in light of the fact that $\BA_\kappa (X)$ can fail (in $\ZFC$) for some perfect Polish spaces $X$. For instance if $X$ is the unit interval $[0, 1]$ then no autohomeomorphism can map a dense set containing the endpoints to one which does not. More generally if $X$ is a manifold with boundary then there will be $\kappa$-dense sets which contain the the boundary and those that do not and the same issue arises.  

Finally we note the following.
\begin{corollary}
    $\BA^- (\mathbb R)$ does not imply $\BA$ ($= \BA (\mathbb R)$ ).
\end{corollary}

\begin{proof}
This is an immediate consequence of the previous corollary and the aforementioned fact that $\BA$ does not follow from $\MA + \neg \CH$. 
\end{proof}

Awkwardly I don't know if the above corollary is true if $\mathbb R$ is replaced by Baire or Cantor space. A more precise question along these lines will be listed in the final section of this paper. 

With regards to the $\mathsf{U}_\kappa$ axioms, the parametrizations are equally entirely superfluous. 
\begin{fact}
Let $\kappa < 2^{\aleph_0}$ be a cardinal. The following are equivalent.
\begin{enumerate}
    \item $\mathsf{U}_\kappa$
    \item $\mathsf{U}_\kappa (X)$ holds for every perfect Polish space $X$. 
    \item $\mathsf{U}_\kappa (2^\omega)$
    \item There is a perfect Polish space $X$ so that $\mathsf{U}_\kappa(X)$ holds.
\end{enumerate}
\label{Uequiv1}
\end{fact}

\begin{proof}
    The equivalence of (1) and (3) is an immediate consequence of Fact \ref{embedintocantorspace}. Moreover (2) implies (3) implies (4) is by definition. Also, since every perfect Polish space is Cantor crowded and $\kappa$ is assumed to be less than the continuum, if $X$ is a perfect Polish space then every subset of $X$ of size $\kappa$ is homeomorphic to a subset of $2^\omega$ of size $\kappa$ and vice versa which proves the equivalence of (3) with both (2) and (4).
\end{proof}

A variant of this proof shows the following.
\begin{fact}
Let $\kappa < \lambda < 2^{\aleph_0}$ be two cardinals. The following are equivalent.
\begin{enumerate}
    \item $\mathsf{U}_{\kappa, \lambda} (X)$ holds for every perfect Polish space $X$. 
    \item $\mathsf{U}_{\kappa, \lambda} (2^\omega)$
    \item There is a perfect Polish space $X$ so that $\mathsf{U}_{\kappa, \lambda} (X)$ holds.
\end{enumerate}
\label{Uequiv2}
\end{fact}
In both Facts \ref{Uequiv1} and \ref{Uequiv2} we note that having either $\kappa$ or $\lambda$ equal to $\cc$ simply gives $\ZFC$-provable axioms. This is because $2^\omega$ is a witness to $\mathsf{U}_{\cc}$ and $X$ itself is a witness to $\mathsf{U}_{\cc}(X)$ for every perfect Polish space $X$. 

\section{Weak Baumgartner Axiom}

In this section short section we observe that $\BA^-_\kappa$ implies some of the more striking consequences of $\BA$. We start with some cardinal arithmetic. It was shown in \cite[Theorem 7.1]{ARS85} that $\BA$ implies that $2^{\aleph_0} = 2^{\aleph_1}$. It turns out a generalization of this fact follows from just $\BA^-$.

\begin{lemma}
    Suppose $\kappa$ is an uncountable cardinal. If $\BA^-_\kappa$ holds then $2^{\aleph_0} = 2^\kappa$. \label{BA-impliesCA}
\end{lemma}

\begin{proof}
    By Lemma \ref{BA-equivs}, $\BA^-_\kappa$ is equivalent to $\BA^-_\kappa(P)$ for any perfect Polish space $P$. Fix such a $P$ and let $\kappa$ be an uncountable cardinal. Let $A \subseteq P$ be $\kappa$ dense. We claim that if $2^\kappa > 2^{\aleph_0}$ then there are $\kappa$ dense subsets of $A$ which are not homeomorphic. Partition $A$ into $\kappa$ many disjoint, countable dense pieces, say $A = \bigcup_{\alpha < \kappa} A_\alpha$ (this is possible by $\kappa$-density). For each $Z \in [\kappa]^\kappa$ let $A_Z = \bigcup_{\alpha \in Z} A_\alpha$. Note that for each such $Z$ the set $A_Z$ is also $\kappa$-dense and for distinct $Z$ and $Z'$ we have $A_Z \neq A_{Z'}$. By $\BA^-_\kappa (P)$ for every $Z \in [\kappa]^\kappa$ there is a homeomorphism $h_Z:A \to A_Z$ and hence by Fact \ref{Lavthm} there are $G_\delta$ subsets $V_Z$ and $W_Z$ of $P$ and a homeomorphism $\hat{h}_Z:V_Z \to W_Z$ extending $h_Z$. But there are only $2^{\aleph_0}$ many triples of a pair of $G_\delta$ subsets of $P$ and a homeomorphism between them so if $2^\kappa > 2^{\aleph_0}$ then there are distinct $Z \neq Z'$ so that $\hat{h}_Z = \hat{h}_{Z'} := \hat{h}$. However this is immediately a contradiction as $\hat{h} \hook A$ would have to homeomorphically map $A$ onto two distinct sets.
\end{proof}

Next we remark that an unpublished result of Medini states that similarly $\BA^-_\kappa$ implies $\mfb \neq \kappa$. In fact, much more is true.  
\begin{theorem}[Medini, Unpublished]
    If $X$ is a Cantor-crowded space and $\BA_\kappa^-(X)$ holds then $\kappa \neq \mfb$ and $\kappa \neq \mfd$.\label{medini}
\end{theorem}

Note that despite Lemma \ref{BA-equivs}, the parametrization is needed in Theorem \ref{medini} as $X$ is not assumed to be a perfect Polish space.

\section{Universal Spaces}
We now turn to the axiom $\mathsf{U}_\kappa$ and its parametrized versions. As noted above $\mathsf{U}_{\aleph_0}$ and $\mathsf{U}_{2^{\aleph_0}}$ are theorems of $\ZFC$. First let us show that these are the only cardinals with these properties provably in $\ZFC$. 

\begin{theorem}
    Let $\kappa < \lambda$ be uncountable cardinals. If $\P$ is the standard forcing to add either $\lambda$ many Cohen reals or $\lambda$ many random reals then $\mathsf{U}_\kappa$ fails in $V[G]$ for any $G \subseteq \P$ which is generic over $V$. \label{Ufailsthm}
\end{theorem}

This theorem will follow immediately from the following lemma.

\begin{lemma}
    Let $\kappa < \lambda$ be cardinals and $\P$ either the standard forcing to add $\lambda$ many Cohen reals or $\lambda$ many random reals. Let $G \subseteq \P$ be $\P$-generic. In $V[G]$ the following holds: for every perfect Polish space $X$ if $U \subseteq X$ has size $\kappa$ then there is a $W \subseteq X$ of size $\aleph_1$ so that every continuous $f:W \to U$ has countable range. \label{Ufailslemma}
\end{lemma}

Note that the conclusion of the above lemma states that in particular for every $\kappa < \lambda$ and $X$ a perfect Polish space we have that $\mathsf{U}_{\aleph_1, \kappa} (X)$ fails and in particular there can be no universal space of size $\kappa$, thus Lemma \ref{Ufailslemma} implies Theorem \ref{Ufailsthm}.

\begin{proof}[Proof of Lemma \ref{Ufailslemma}]
Let $\kappa < \lambda$ be uncountable cardinals. First note that if we add $\lambda$ many reals (in any fashion which preserves cardinals) then for any perfect Polish space $X$ if $Z \subseteq X$ has size $\kappa$ then it is zero dimensional and thus can be treated as a subspace of $2^\omega$ by Fact \ref{embedintocantorspace}. Since the existence of continuous functions between subspaces does not depend on the ambient space, we can therefore restrict our attention to the case where $X = 2^\omega$.

We begin with the case of Cohen forcing. Let $\mathbb C_\lambda$ denote the standard forcing to add $\lambda$ many Cohen reals. Let $G \subseteq \mathbb C_\lambda$ be generic and denote the generic Cohen reals by $\{c_\alpha\; | \; \alpha \in \lambda\}$. Work in $V[G]$. Now let $U \subseteq X$ of size $\kappa$. Since $\kappa < \lambda$, by the ccc there is a a set $A \subseteq \lambda$ of size $\kappa$ so that $U \in V[c_\alpha \; | \; \alpha \in A]$. Let $B \subseteq \lambda \setminus A$ be uncountable. We claim that $W:= \{c_\beta \; | \; \beta \in B\}$ is as needed. By mutual genericity we can treat every real in $W$ as being Cohen generic over $V[c_\alpha \; | \; \alpha \in A]$. Thus it suffices to show the following claim:

\begin{claim}
    Suppose $I$ is an uncountable index set, $\mathbb C_I$ is the Cohen forcing for adding Cohen reals indexed by $I$ and $H \subseteq \mathbb C_I$ is generic over $V$ with $H = \{c_i\; | \; i \in I\}$ the corresponding Cohen reals. Let $J \subseteq I$ be an uncountable subset in $V[H]$. No continuous map $f$ from the set $\{c_i\; | \; i \in J\}$ into the ground model can have uncountable range.
\end{claim}

\begin{proof}
    Let $f\in V[H]$ be a continuous map from the set $\{c_i\; | \; i \in J\}$ into the ground model. By Fact \ref{Lavthm} there are $G_\delta$ subsets $E, F \subseteq 2^\omega$ and a lift $\hat{f}:E \to F$. Since these are all coded by a real we have that $E, F, \hat{f} \in V[c_k\; | \; k \in K]$ for some {\bf countable} $K \subseteq I$. By standard properties of Cohen forcing, working in $V[H]$, if some $x \in F \cap V$ is such that $f(c_j) = \hat{f}(c_j) = x$ for any $j \in J \setminus K$ then it must be the case that $V[c_k\; | \;k \in K] \models$ ``$\hat{f}^{-1}(x)$ is non meager" since every such $c_j$ is Cohen over $V[c_k\; | \;k \in K]$ and hence cannot be in any Borel set coded into $V[c_k\; | \;k \in K]$ which is meager. But now if $x \neq y \in V \cap F$ we have obviously that $\hat{f}^{-1}(x) \cap \hat{f}^{-1}(y) = \emptyset$, and there cannot be an uncountable family of disjoint, non-meager, closed sets in any Polish space. It follows that there are most countably many $x \in V \cap F$ so that if $j \in J \setminus K$ then $f(c_j) = x$. But since $K$ is countable the result then follows. 
\end{proof}

The case of random forcing is nearly identical. We sketch it for completeness and leave the reader to fill in the details based on the case of Cohen forcing above. Recall that adding $\lambda$ many random reals, concretely means forcing with the positive measure Borel sets in the measure algebra $2^{\lambda \times \omega}$ where we give the product measure on the uniform measure on $2$, see \cite{KenST}. Denote this by $\mathbb B_\lambda$. More generally if $I$ is an index set let $\mathbb B_I$ be the associated random forcing. For each $i \in I$ let $r_i$ denote the canonical $i^{\rm th}$ random real in the extension. As in the case of Cohen forcing, showing the lemma for random forcing reduces to the following claim.

\begin{claim}
    
Suppose $I$ is an uncountable index set and $H \subseteq \mathbb B_I$ is generic over $V$ with $H = \{r_i\; | \; i \in I\}$ the corresponding random reals. Let $J \subseteq I$ be an uncountable subset in $V[H]$. No continuous map $f$ from the set $\{r_i\; | \; i \in J\}$ into the ground model can have uncountable range.
\end{claim}

\begin{proof}
    As before any continuous $f: \{r_i\; | \; i \in J\} \to V$ lifts to a continuous $\hat{f}:E \to F$ with $\hat{f}, E. F \in V[r_k\; | \; k \in K]$ for countable $K$. Now, if $i \in J \setminus K$ and in $V[H]$ $f(r_i) = x$ for some $x \in V$ then in $V[r_k\; | \; k \in K]$ we have that $\hat{f}^{-1}(x)$ has positive measure and therefore there are at most countably many $x$ with this property. The result now follows as before.
\end{proof}

\end{proof}

We now prove, despite that fact that $\mathsf{U}_\kappa$ can fail in the Cohen and random models, these axioms have neither of the consequences of $\BA^-_\kappa$ discussed in the previous section. The first of these is actually an easy consequence of \cite[Theorem 4.7]{independenceresults}, due to Shelah\footnote{To be clear, given \cite[Theorem 4.7]{independenceresults} the consequence is easy. The proof of \cite[Theorem 4.7]{independenceresults} is not easy.}, see also \cite{Burke2009}.

\begin{theorem}[Essentially Shelah, See \cite{independenceresults}]
    It is consistent that $\cc = \aleph_2$ and $\mathsf{U}_{\aleph_1}$ holds but ${\rm non}(\Me) = \aleph_1$.\label{shelah}
\end{theorem}
Since $\mfb \leq \non (\Me)$ in $\ZFC$ it follows that in the above model $\mfb =\aleph_1$ as well.

\begin{proof}
    (The proof of) theorem 4.7 of \cite{independenceresults} states that it is consistent with $\cc = \aleph_2$ that $\non (\Me) = \aleph_1$ and given any $A, B \subseteq \mathbb R$ which are both non meager in every open interval and $\aleph_1$-dense, they are in fact isomorphic as linear orders. In particular they are homeomorphic. Since $\non (\Me) = \aleph_1$ in this model, given any $\aleph_1$-sized $C \subseteq \mathbb R$ it can be enlarged to an $\aleph_1$-dense $B \subseteq \mathbb R$ which is non meager in every open interval. Thus in this model any $\aleph_1$-sized set is embeddable (as a linear order even) into the unique (up to isomorphism) $\aleph_1$-dense linear order which is nowhere meager. This latter set is the universal $Z \subseteq \mathbb R$ witnessing $\mathsf{U}_{\aleph_1}$.
\end{proof}

Moving on to cardinal arithmetic for now we show that the conclusion of Lemma \ref{BA-impliesCA} does not follow from the $\mathsf{U}_\kappa$ axioms. Let us state precisely what we will show before getting into it.

\begin{theorem}
    Assume $\GCH$. Let $\kappa < \lambda$ be cardinals so that $\lambda$ has cofinality greater than $\aleph_1$. There is a forcing extension in which $2^{\aleph_1} = \lambda$, $2^{\aleph_0} = \kappa^+$ and $\mathsf{U}_{\kappa}$ holds. In particular it is consistent with $\mathsf{U}_{\aleph_1} + \neg \CH$ that $2^{\aleph_0} < 2^{\aleph_1}$. \label{nocardarith}
\end{theorem}

In order to prove this theorem we need a forcing notion due to Medini, see \cite[pp. 138-139]{Medini15}. 
\begin{definition}(Medini)
   Let $\kappa \leq \cc$ be a cardinal and let $A, B \subseteq 2^\omega$ be $\kappa$-dense. Fix partitions $\{A_\alpha\}_{\alpha \in \kappa}$ and $\{B_\alpha\}_{\alpha \in \kappa}$ so that for each $\alpha < \kappa$ we have that $A_\alpha$ and $B_\alpha$ are countable and dense and $A_\alpha \cap A_\beta = \emptyset = B_\alpha \cap B_\beta$ for all distinct $\alpha, \beta \in \kappa$. The forcing notion $\P_{A, B}$ consists of pairs $p = (\pi_p, f_p)$ so that the following hold: 
   
   \begin{enumerate}
       \item $\pi_p$ is a permutation of $2^n$ for some $n = n_p \in \omega$
       \item $f_p$ is a finite partial injection from $A$ to $B$ so that for all $\alpha$ we have that $x \in A_\alpha \cap {\rm dom}(f_p)$ if and only if $f_p(x) \in B_\alpha$. 
       \item $f_p$ respects $\pi_p$ i.e. for all $x \in {\rm dom}(f_p)$ we have that $\pi_p (x\hook n) = f_p(x) \hook n$. 
   \end{enumerate}

   The order is as follows: we let $q \leq p$ if and only if:
   \begin{enumerate}
       \item $f_q \supseteq f_p$ and 
       \item $n_p = n_q$ and $\pi_p = \pi_q$ or $n_q > n_p$ and for all $s \in 2^{n_q}$ we have that $\pi_q(s)\hook n_p = \pi_p (s\hook n_p)$.
   \end{enumerate}
   \label{medinidef}
\end{definition}

This partial order is a more streamlined version of the original one due to Baldwin in Beaudoin from \cite[Lemma 3.1]{BaldwinBeaudoin89}.

The following facts, which may be black boxed for the sake of the proof are in fact all we need.

\begin{fact}
    Let $A, B \subseteq 2^\omega$ be $\kappa$-dense. The following hold:
    \begin{enumerate}
        \item $|\P_{A, B}| = \kappa$
        \item $\P_{A, B}$ is $\sigma$-centered
        \item $\P_{A, B}$ forces the existence of an autohomeomorphism $h:2^\omega \to 2^{\omega}$ so that $h``A = B$
    \end{enumerate}
\end{fact}

We can now prove Theorem \ref{nocardarith}.
\begin{proof}[Proof of Theorem \ref{nocardarith}]
    Assume $\GCH$ and fix $\kappa < \lambda$ cardinals and assume $\lambda$ has cofinality larger than $\aleph_1$. Let $\P_0$ be the standard forcing to force $2^{\aleph_1} = \lambda$ and $2^\omega = \kappa$ simultaneously. Let $V_0$ be a generic extension of $V$ by $\P_0$. Let $A = B_0$ be the reals of $V_0$. Now define a finite support iteration $\{\P_\alpha, \dot{\Q}_\alpha \; | \; \alpha < \kappa^+\}$ inductively as follows. Assume $\P_\alpha$ has been defined, $G_\alpha \subseteq \P_\alpha$ is generic and let $V_\alpha := V[G_\alpha]$. Inductively assume that $V_\alpha \models 2^{\aleph_0} = \kappa + 2^{\aleph_1} = \lambda$, that $\P_\alpha/\P_0$ is ccc and therefore that $A$ is $\kappa$-dense. Now let $B_\alpha = 2^\omega \cap V_\alpha$ and $\dot{\Q}_\alpha^{G_\alpha} = \Q_\alpha$ the forcing $\P_{A, B_\alpha}$. Note that this forcing is ccc and has size $\kappa$ hence in $V_\alpha^{\Q_\alpha}$ we have that $2^{\aleph_0} = \kappa + 2^{\aleph_1} = \lambda$. The first inequality follows since inductively we have that $\kappa^\omega = \kappa$ and hence there are $\kappa$-many nice names for reals for any Medini forcing as in Definition \ref{medinidef}. Now clearly $\P_{\kappa^+}$ forces $2^{\aleph_0} = \kappa^+ +2^{\aleph_1} = \lambda$. To see that $\mathsf{U}_{\kappa}$ holds, observe that if $C \subseteq 2^\omega$ is a set of reals of size $\kappa$ in $V_{\kappa^+}$ then there is an $\alpha < \kappa^+$ so that $C \subseteq V_{\alpha} \cap 2^\omega := B_\alpha$. But $B_\alpha$ was forced to be homeomorphic to $A$ and hence $C$ is homeomorphic to a subset of $A$, thus $A$ is the universal set desired.
\end{proof}

We remark that an immediate consequence of Theorem \ref{nocardarith} is the following corollary.
\begin{corollary}
    For no uncountable cardinal $\kappa$ does $\mathsf{U}_{\kappa}$ imply $\BA^-_\kappa$. 
\end{corollary}

It would be nice to obtain this Corollary using Theorem \ref{shelah} as well however in that case I only know what to do when the continuum is $\aleph_2$. In particular, to the best of my knowledge, it is open if the Medini forcing notions used in the proof of Theorem \ref{medini} can add a dominating real (when $A$ and $B$ are both, say, nowhere meager). 

The proofs of Theorems \ref{shelah} and \ref{nocardarith} are in fact compatible. An inspection of the proof of Theorem \ref{shelah} in \cite{independenceresults} shows that we just need $\diamondsuit$ to hold over the ground model. Thus, by assuming $\diamondsuit$ in the ground model and applying Shelah's forcing as opposed to Medini's we in fact get the following consistency result. This strengthens Theorem \ref{noconsequencesofU1} from the introduction.
\begin{theorem}
    It is consistent that $\mathsf{U}_{\aleph_1}$ holds, $\non (\Me) = \aleph_1 < \cc =\aleph_2$ and $2^{\aleph_1} = \lambda$ for any cardinal $\lambda$ of cofinality larger than $\aleph_1$. \label{noconsequencesofU2}
\end{theorem}

We note that we still can only get that some of the nowhere meager linear orders are isomorphic as we can still only perform the iteration for $\omega_2$ steps. 

\section{Spaces which Avoid One Another}
In this section we want to take a more fine grained approach to the homeomorphism types of various $\kappa$-dense subsets of $2^\omega$. In particular, motivated by the results of the previous section let us make the following definitions.

\begin{definition}
    Let $X$ and $Y$ be topological spaces spaces\footnote{In practice we will be interested in subspaces of $2^\omega$ though the definition does not require this.} with $|X| = \kappa$ for some uncountable cardinal $\kappa$. 
    \begin{enumerate}
    \item 
    We say that $X$ {\em avoids} $Y$ if no $Z \subseteq X$ of size $\kappa$ is homeomorphic to a subspace of $Y$.
        \item We say that $X$ {\em strongly avoids} $Y$ if for each $Z \subseteq X$ of size $\kappa$ and every $f:Z \to Y$ continuous the image of $f$ has size ${<}\kappa$. 
        \item We say that $X$ {\em totally avoids} $Y$ if for every $Z \subseteq X$ of size $\kappa$ and every $f:Z \to Y$ continuous the image of $f$ is countable.
    \end{enumerate}
\end{definition}

Note that totally avoiding implies strongly avoiding implies avoiding. Moreover, totally avoiding and strongly avoiding are the same in the case of $\kappa = \aleph_1$ but otherwise totally avoiding a space is seemingly stronger than strongly avoiding it. Observe the following, which is immediate from the definitions.
\begin{observation}
    Let $\kappa < \lambda$ be cardinals. 
\begin{enumerate}
    \item 

    If $\mathsf{U}_{\kappa, \lambda}(X)$ holds for some perfect Polish space $X$ then there is a subspace $Z \subseteq X$ of size $\lambda$ which no subspace $Y \subseteq X$ of size $\kappa$ avoids.
    \item 
    If $\BA_\kappa^-$ holds then no $\kappa$-dense subsets of any perfect Polish space avoid one another. 

    \end{enumerate}
\end{observation}

The point is that what was shown above is actually the following.

\begin{lemma}
Let $\kappa$ be a cardinal and $\P$ be either the standard forcing to add $\kappa$ Cohen reals or $\kappa$ random reals. If $G \subseteq \P$ is generic over $V$ then any uncountable subset of the generic reals given by $G$ (as thought of in any perfect Polish space $P$ coded into the ground model) totally avoids the ground model elements of $P$.
\end{lemma}

To some extent the above lemma can be recovered axiomatically.
\begin{proposition}
 Assume $\cov (\Me) = \cc$ or $\cov (\Null) = \cc$. There are are $\cc$-dense $X$ and $Y$ subsets of $2^\omega$ so that $X$ strongly avoids $Y$.   \label{covMprop}
\end{proposition}
    
\begin{proof}
  We do the case of $\cov (\Me) = \cc$. The case of $\cov (\Null) = \cc$ is almost identical replacing ``Cohen reals" with ``random reals" and ``meager" with ``measure zero". The relevant fact about $\cov (\Me) = \cc$ is that this implies that if $\kappa$ is an uncountable cardinal and $M \prec H_\kappa$ is a model of size ${<} \cc$ then there are two mutually generic Cohen reals over $M$. Now let $\{M_\alpha \; | \; \alpha < \cc\}$ be a continuous $\in$-sequence of elementary submodels of $H_{\omega_1}$ which collectively cover $H_{\omega_1}$. By elementarity plus the fact listed two sentences ago, for each $\alpha < \cc$ there two reals $x_\alpha, y_\alpha \in M_{\alpha + 1}$ which are mutually Cohen generic over $M_\alpha$ (and hence $M_\beta$ for all $\beta < \alpha$ as well). Let $X = \{x_\alpha \; | \; \alpha < \cc\}$ and $Y = \{y_\alpha \; | \; \alpha < \cc\}$. Clearly these can be assumed to be $\cc$-dense since any translation of a Cohen real is still Cohen. 

Let us show that $X$ strongly avoids $Y$. Let $Z \subseteq X$ have size continuum and let $f:Z \to Y$ be continuous. We want to show that $f$ has range of cardinality less than $\cc$. Let $\hat{f}$ be the lift of $f$ to some function on a $G_\delta$ subset $W$ of $2^\omega$. Since $\hat{f}$ and $W$ are coded by reals there is an $\alpha \in \cc$ so that $\hat{f}, W \in M_\alpha$. Fix the least such. If the range of $f$ has cardinality $\cc$ then its restriction to any tail of elements of $Z$ does as well - i.e. without loss of generality we may assume that if $x_i \in Z$ then $i > \alpha$. Let us fix $A \subseteq \cc$ so that $Z = \{x_i\; | \; i \in A\}$. There are now a few cases to consider.

\noindent \underline{Case 1}: There is an $i \in A$ so that $f(x_i) = y_i$. In this case then we can recover $y_i$ from $\hat{f}$ and $x_i$ and therefore in particular $y_i \in M_i[x_i]$ which contradicts mutual genericity.

\noindent \underline{Case 2}: There is an $i \in A$ and a $j > i$ so that $f(x_i) = y_j$. In this case again we can recover $y_j$ from $\hat{f}$ and $x_i$ and therefore in particular $y_j \in M_{i+1}$, which it is not by assumption. 

\noindent \underline{Case 3}: Cases 1 and 2 fail i.e. for every $i \in A$ if $f(x_i) = y_j$ then $j < i$. Thus for any such $i$ and $j$ we have that there is a Cohen real, namely $x_i$, over $M_{j + 1}$ in the Borel set $\hat{f}^{-1}\{y_j\}$, which is coded into $M_{j+1}$ and therefore this set is non-meager. It follows in this case that there can only be countably many $y_j$ in the image of $f$. 

Putting all of this together we get that the image of $f$ contains the ${<}\cc$-many points which appear in $M_\alpha$ alongside at most the countably many points from case 3 and in particular the image of $f$ is not of size continuum. 
\end{proof}

Some set theoretic hypothesis is needed to obtain the conclusion of Proposition \ref{covMprop} above. This is because of a result essentially due to Miller from \cite{Miller83}, though not in this language.
\begin{theorem}[Essentially Miller, \cite{Miller83}]
   In the Sacks model there are no subspaces of $2^\omega$, of size continuum which strongly avoid one another.\label{miller}
\end{theorem}

We note that in some sense the Sacks model is a canonical model of the failure of $\cov(\Me) + \cov(\Null) = \cc$. 

\begin{proof}
    In \cite{Miller83} the main result is that in the Sacks model every $X \subseteq 2^\omega$ of size continuum continuously surjects onto $2^\omega$. In particular if $X, Y \subseteq 2^\omega$ have size continuum then there is a continuous surjection $f:X \to 2^\omega$ and therefore for some $Z \subseteq X$ we have that $f\hook Z$ is a continuous bijection onto $Y$. 
\end{proof}

Nevertheless, for sake of completeness let us show that at least some of these ideas in fact can be contrived from simple $\ZFC$ combinatorics.
\begin{proposition}
      ($\ZFC$) There are continuum dense subsets $X, Y \subseteq 2^\omega$ so that $X$ avoids $Y$. \label{avoid}
\end{proposition}

In the above $2^\omega$ can be replaced with any uncountable Polish space (or even analytic space) since each such contains a closed copy of $2^\omega$ and totally avoiding/strongly avoiding/avoiding do not depend on the ambient space.

\begin{proof}
    Enumerate $2^\omega = \{z_\alpha \; | \; \alpha \in 2^{\aleph_0}\}$ and enumerate all homeomorphisms between $G_\delta$-subsets of Cantor space as $\{f_\alpha\; | \; \alpha \in 2^{\aleph_0}\}$. We will construct $X = \{x_\alpha\; | \; \alpha \in 2^{\aleph_0}\}$ and $Y = \{y_\alpha \; | \; \alpha \in 2^{\aleph_0}\}$ by recursion, diagonalizing against all possible lifts. It's important to note here that $G_\delta$ subsets of Polish spaces are Polish and hence the image under any continuous function of such a set will be analytic and therefore either countable or contain a perfect set. Since we are only considering homeomorphisms, it follows in particular that the image of any if the $f_\alpha$'s will contain a perfect subset, a fact we will use.

    Fix a countable basis for $2^\omega$ as $\{U_n\; | \; n < \omega\}$ (for instance the standard one defined by finite initial segments would work) and partition $2^{\aleph_0}$ into countably many pieces each of size continuum, say $\{C_n\}_{n < \omega}$ (this will be used to ensure that the $X$ and $Y$ we construct are continuum dense). Inductively we will define sets $X_\alpha$ and $Y_\alpha$ for $\alpha \in 2^{\aleph_0}$ so that 
    \begin{enumerate}
        \item $|X_\alpha| = |Y_\alpha|$ has size $|\alpha + 1|$ and enumerate them as $\{x_i\; | \; i \in \alpha +1\}$ and $Y_\alpha = \{y_i\; | \; i \in \alpha +1 \}$.

        \item $X_\alpha \subsetneq X_\beta$ and $Y_\alpha \subsetneq Y_\beta$

        \item For all $\beta \leq \alpha$, if $x \in X_\alpha \setminus X_\beta$ then $f_\beta(x) \notin Y_\alpha \setminus Y_\beta$. In other words $[f_\beta``(\bigcup_{\alpha \in \cc} X_\alpha) ]\cap (\bigcup_{\alpha \in \cc} Y_\alpha) \subseteq Y_\beta$.
        
        \item If $\alpha \in C_n$ then $X_\alpha \setminus \bigcup_{\beta < \alpha} X_\beta$ and $Y_\alpha \setminus \bigcup_{\beta < \alpha}Y_\beta$ will intersect $U_n$.
    \end{enumerate}

    Supposing we can accomplish this let $X = \bigcup_{\alpha \in 2^{\aleph_0}} X_\alpha$ and $Y = \bigcup_{\alpha \in 2^{\aleph_0}} Y_\alpha$. Note that both of these have size continuum by (1) and in fact are continuum dense by (4).  Moreover, $X$ avoids $Y$. To see this, suppose towards a contradiction that $f: Z \to Y$ is a homeomorphism onto its image for some $Z \subseteq X$ of size $2^{\aleph_0}$. By Fact \ref{Lavthm} there is a $\beta < 2^{\aleph_0}$ so that $f = f_\beta \hook Z$. This is a contradiction though because by (3) we have that $f_\beta `` Z \subseteq Y_\beta$ but $Y_
    \beta$ has size less than continuum and in particular cannot be homeomorphic to $Z$. 
    
    Thus it suffices to show that the construction described above can be performed. Without loss of generality we can assume that $0 \in C_0$ and $z_0 \in U_0$ (if not we can easily rearrange our indexing to make it so). At stage $0$ then let $X_0 = \{z_0\}$. Let $Y_0$ be a singleton consisting of any point in $\mathsf{U}_0$ other than $f_0(z_0)$. Clearly this satisfies (1) - (4). Now assume that $\gamma$ is an ordinal less than $2^{\aleph_0}$ and $X_\alpha$ and $Y_\alpha$ have been constructed for all $\alpha < \gamma$ satisfying (1) through (4). Let $n$ be such that $\gamma \in C_n$. We claim that there is an $x \in \mathsf{U}_n \setminus (\bigcup_{\alpha < \gamma} X_\alpha)$ so that $f_\beta (x)  \notin Y_\alpha$ for any $\alpha < \gamma$ and $\beta \leq \gamma$. Since $\bigcup_{\alpha < \gamma} X_\alpha$ has size less than continuum there are continuum many points in $U_n \setminus (\bigcup_{\alpha < \gamma} X_\alpha)$. Suppose towards a contradiction that for each such $x$ there is a $\beta$ and an $\alpha$ so that $f_\beta (x) \in Y_\alpha$. Since each $Y_\alpha$ has cardinality $|\alpha + 1|$ we get that $\bigcup_{\alpha < \gamma} Y_\alpha$ has cardinality $|\leq \gamma + 1|$. Thus in particular there is a fixed $y \in \bigcup_{\alpha < \gamma} Y_\alpha$ for which the set of $x \in U_n \setminus (\bigcup_{\alpha < \gamma} X_\alpha)$ with $f_\beta (x) = y$ for some $\beta \leq \gamma$ has size continuum. But since $\gamma < \cc$ there are two points in this set, say $x$ and $x'$ for which the witnessing $f_\beta$ is the same. As such $f_\beta$ is not injective and in particular is not a homeomorphism, which is a contradiction. 
    
    Now let $X_\gamma = \bigcup_{\alpha < \gamma} X_\alpha \cup \{z_i\}$ where $i$ is least so that $z_i \in U_n \setminus \bigcup_{\alpha < \gamma} X_\alpha$ and $f_\beta (z_i) \notin \bigcup_{\alpha < \gamma} Y_\alpha$ for all $\beta \leq \gamma$. Finally note that the set $Z_\gamma :=\{f_\beta (x) \; | \; \beta \leq \gamma \, x \in X_\gamma\}$ has size less than continuum and hence there is a $y \in U_n$ not in $Z_\gamma$. Let $j$ be least so that $z_j \in \mathsf{U}_n \setminus Z_\gamma$ and let $Y_\gamma = \bigcup_{\alpha < \gamma} Y_\alpha \cup \{z_j\}$. It is clear that (1) to (4) are satisfied.  
\end{proof}

\begin{remark}
    Note that if $X$ avoids $Y$ then neither can contain a perfect set since in this case one would be homeomorphic to a subset of the other. In particular no analytic set can avoid another analytic set. The above proof however can be done so that $X$ and $Y$ are $\Pi^1_1$ if $V=L$ using Miller's method from \cite{Miller89}. This is simply because the $z_i$ and $z_j$ chosen in the proof (at the inductive step to build the next stage for $X_\gamma$ and $Y_\gamma$) can be made so to code the entire construction up to that point thus satisfying the ``encoding lemma", see e.g. \cite[Lemma 8.24]{Miller89} and the associated discussion, also in Section 7 of the same paper for more details. The reader familiar with this method should be able to reconstruct the details. We neglect to do so here as it would take us too far afield. 
\end{remark}

Whether the above can be improved from ``continuum dense" to $\kappa$-dense for any $\kappa < \cc$ is independent of $\ZFC$. Indeed under $\mfp = \cc$ all $\kappa$-dense subsets of $2^\omega$ are homeomorphic as discussed above and in particular fail to strongly avoid one another while the proof of Theorem \ref{Ufailsthm} gives that for each $\kappa < \cc$ there are large families of $\kappa$-dense sets which all simultaneously totally avoid one another in the Cohen and random models. In short the following holds.
\begin{proposition}
    If $\kappa < \mfp$ then no pair of $\kappa$-dense subsets of $2^\omega$ avoid one another while it is consistent that for all uncountable $\kappa \leq \cc$ there are $X$ and $Y$, $\kappa$-dense subsets of $2^\omega$ so that $X$ totally avoids $Y$. \label{nostrongly avoid}
\end{proposition}

Also, combining Proposition \ref{avoid} with Theorem \ref{miller} of Miller, it is independent of $\ZFC$ also whether the above can be improved from ``avoids" to ``strongly avoids". In particular we have the following. 
\begin{corollary}
    Consistently there are $\cc$-dense $X, Y \subseteq 2^\omega$ so that $X$ avoids $Y$ but does not strongly avoid $Y$. 
\end{corollary}

The utility of these ideas is that they play well with finite support iterations.

\begin{theorem}
    Let $P$ be a perfect Polish space, $X, Y \subseteq P$, $|X| = \kappa$ has uncountable cofinality and let $\delta$ be an ordinal. Assume $\langle \P_\alpha, \dot{\Q}_\alpha \, | \, \alpha < \delta \rangle$ is a finite support iteration of ccc forcing notions.
    \begin{enumerate}
     \item If $X$ avoids $Y$ and for each $\alpha$ we have that $\forces_\alpha$ ``$\dot{\Q}_\alpha$ forces that $\check{X}$ avoids $\check{Y}$" then $\forces_\delta$``$\check{X}$ avoids $\check{Y}$".
     
        \item If $X$ strongly avoids $Y$ and for each $\alpha$ we have that $\forces_\alpha$ ``$\dot{\Q}_\alpha$ forces that $\check{X}$ strongly avoids $\check{Y}$" then $\forces_\delta$``$\check{X}$ strongly avoids $\check{Y}$".

        \item If $X$ totally avoids $Y$ and for each $\alpha$ we have that $\forces_\alpha$ ``$\dot{\Q}_\alpha$ forces that $\check{X}$ totally avoids $\check{Y}$" then $\forces_\delta$``$\check{X}$ totally avoids $\check{Y}$".
    \end{enumerate}
    \label{iteration}
\end{theorem}
In short the above states that all three versions of avoiding are preserved by finite support iterations of ccc forcing notions.

\begin{proof}
    We prove all three cases simultaneously as they are all similar. Fix $P$ a perfect Polish space and $X, Y \subseteq P$ and assume $X$ avoids (strongly avoids, totally avoids) $Y$. Suppose $\langle \P_\alpha, \dot{\Q}_\alpha \, | \, \alpha < \delta \rangle$ is a finite support iteration of ccc forcing notions and for each $\alpha$ we have that $\forces_\alpha$ ``$\dot{\Q}$ forces that $\check{X}$ avoids (strongly avoids, totally avoids) $\check{Y}$". The proof is by induction on $\delta$. The successor step is by assumption so we focus on the limit step. We fix $\kappa$ so that $|X| = \kappa$ and note that throughout the proof since we are only considering ccc forcing extensions $\kappa$ is not collapsed so there is no ambiguity between forcing extensions on the size of $\kappa$ sized subsets of $X$. There are two cases.

    \noindent \underline{Case 1}: $\delta$ has uncountable cofinality. Let $H_\delta$ be $\P_\delta$-generic\footnote{The usual notation for this object is unfortunately $G_\delta$ which however in this proof needs to be reserved for the countable intersection of open sets.} and work in $V[H_\delta]$ momentarily. Suppose $Z \subseteq X$ has size $\kappa$ and let $f:Z \to Y$ be continuous. By Fact \ref{Lavthm} there are $G_\delta$ subsets $W_X, W_Y \subseteq P$ and a continuous $\hat{f} \supseteq f$ so that $\hat{f}:W_X \to W_Y$. Moreover if $f$ were forced to be a homeomorphism then $\hat{f}$ is as well. By the ccc plus finite support there is an $\alpha < \delta$ so that $\hat{f}, W_X, W_Y \in V[H_\alpha]$ with $H_\alpha = H_\delta \cap \P_\alpha$. Note by induction we know in $V[H_\alpha]$ that $X$ avoids (strongly avoids, totally avoids) $Y$. 

    Work in $V[H_\alpha]$ and let $\dot{Z}$ and $\dot{f}$ be $\P_\delta/H_\alpha$ names for $Z$ and $f$. By relativizing to some condition in $H_\delta / H_\alpha$ if necessary, without loss we can assume that the maximal condition of $\P_\delta / H_\alpha$ forces that $\dot{Z} \subseteq W_X$ and $\dot{f} \subseteq \hat{f}$. If there is a $p \in \P_{\delta}/H_\alpha$ forcing $x \in X$ to be in $\dot{Z}$ then by Schoenfield absoluteness already $x \in W_X$. Similarly if there is a $p \in \P_{\delta}/H_\alpha$ forcing $\dot{f}(x) = y$ for some $x \in X$ and $y \in Y$ then already $y \in W_Y$ and $\hat{f}(x) = y$. Moreover note that since $\dot{Z}$ is forced to have size $\kappa$, the set of such $x$ also has size $\kappa$. 
    
    To conclude, suppose first we're in the case of avoiding and suppose that $\dot{f}:\dot{Z} \to Y$ were forced to be a homeomorphism of its image. Then by the above we would have already in V[$H_\alpha]$ that $\hat{f} \hook X \cap f^{-1} [Y]$ is a homeomorphism from a subset of $X$ of size $\kappa$ into $Y$ in $V[H_\alpha]$ contradicting the fact that in that model $X$ avoids $Y$. The other two cases are almost verbatim.
    
    \noindent \underline{Case 2}: $\delta$ has countable cofinality. Let $\delta_n \nearrow \delta$ be a strictly increasing sequence of ordinals with $n \in \omega$. As in case 1 let $\dot{Z}$ name a subset of $X$ of size $\kappa$ and let $\dot{f}:\dot{Z} \to \check{Y}$ name a continuous function. Let $H_\delta$ be $\P_\delta$-generic and for each $n < \omega$ let $H_n = H_\delta \cap \P_{\delta_n}$ be the $\P_{\delta_n}$-generic given by the subforcing. Work momentarily in $V[H_\delta]$. Let $f_n$ be the partial function from $X$ to $Y$ given by $f_n(x) = Y$ if and only if there is a condition $p \in H_n$ forcing that $\dot{f}(\check{x}) = \check{y}$. Note that by the finite support we have that $\dot{f}^{H_\delta} =\bigcup_{n < \omega} f_n$. Moreover observe that $f_n \in V[H_n]$, where, by inductive assumption, $X$ and $Y$ still avoid (strongly avoid, totally avoid) one another. Since $\kappa$ has uncountable cofinality, there is an $n<\omega$ so that $f_n$ has domain of size $\kappa$. In the case of avoid, this $f_n$ must also be a homeomorphism of its image, which is a contradiction to the inductive hypothesis. In the cases of strongly and totally avoiding, we note that similarly there must be an $m \geq n$ so that $f_m$ has range of size $\kappa$ (in the case of strongly avoiding) or simply range uncountable (in the case of totally avoiding). This $f_m$ then witnesses the same contradiction. 
    \end{proof}

    This theorem allows us to prove some separations - particularly that even the full Baumgartner's axiom does not allow even the $U$ axiom to hold at a higher cardinal - there is no ``step up" by gluing together e.g. $\aleph_1$-dense sets to get a universal set for $\aleph_2$-dense sets. Towards this we need one simple lemma.

    \begin{lemma}
        Let $P$ be a perfect Polish space and suppose $X, Y \subseteq P$ with $|X| = \kappa > \aleph_0$ for some cardinal $\kappa$. If $\P$ is a forcing notion of size less than ${\rm cf}(\kappa)$ and forces that $X$ does not avoid/strongly avoid/totally avoid $Y$ then (already in $V$), we have that $X$ does not avoid/strongly avoid/totally avoid $Y$. 
    \end{lemma}

    \begin{proof}
        We just prove the case of avoiding, the other two being very similar. The point is simply that if $\P = \{p_\alpha \; | \; \alpha < \mu < {\rm cf}(\kappa)\}$ and forces the existence of some $\dot{f}:\dot{Z} \to Y$ continuous with uncountable image then there is some $\alpha < \mu$ deciding a $\kappa$-sized piece of $Z$ whose image is already uncountable by a simple counting argument and this gives a counter example to totally avoiding in the ground model.
    \end{proof}

We are now ready to prove the ``no step up" theorem.
        \begin{theorem}
    Assume $\GCH$. Let $\aleph_1 \leq \kappa < \mu$ be uncountable cardinals with $\mu$ regular. There is a ccc forcing extension in which $2^{\aleph_0} = \mu$, $\BA_{\kappa '}(2^\omega)$ holds for all $\kappa ' \in [\aleph_1, \kappa]$ but $\mathsf{U}_{\lambda, \lambda'} (2^\omega)$ fails for all $\kappa < \lambda < \lambda ' < \mu$. \label{nostepup}
\end{theorem}

The same proof works for $\baire$ instead of $2^\omega$ and, in the case of $\kappa = \aleph_1$ also for $\mathbb R$ - though a little more preparation of the ground model \`{a} la e.g. \cite[Theorem 9.2]{ARS85} is needed to allow for large continuum. These can even be all forced simultaneously in one model. For the sake of readability we will simply prove the consistency of $\BA_{\aleph_1}(2^\omega) + \neg \mathsf{U}_{\aleph_2} + 2^{\aleph_0} = \aleph_3$. The general case is a straightforward generalization using standard arguments. 
    \begin{proof}
        Assume $\GCH$ and perform a finite support iteration of length $\aleph_3$ alternating between adding $\aleph_2$ Cohen reals and forcing $\aleph_1$-dense subsets of to be homeomorphic by Medini's forcing from Definition \ref{medinidef}. Clearly careful bookkeeping will ensure that $\BA_{\aleph_1}(2^\omega)$ will hold and this can be done in $\aleph_3$ steps so $2^{\aleph_0} = \aleph_3$ will hold as well. 
        
        Suppose now towards a contradiction that $\mathsf{U}_{\aleph_2}$ holds. By the ccc plus finite support there is an initial stage of the iteration in which the universal set $Z$ was added. By enlarging this $Z$ we can assume that it is the reals of some intermediate model. But then the $\aleph_2$ Cohens that were added later totally avoid these ground model reals and this will be preserved by Theorem \ref{iteration}, which is a contradiction completing the proof.
    \end{proof}

\section{Open Questions}

We feel that this general view of ``Baumgartner type topological axioms" is underdeveloped in the literature and therefore we are intentionally generous in the number of questions we ask in a shameless plea to incite more researchers to consider them. In what follows we list these questions roughly in the order the relevant ideas appeared in the paper. 


The first of these comes from the study of the weak Baumgartner axioms for different cardinalities. For instance, the following basic question is open.
\begin{question}
    If $\kappa < \lambda$ does $\BA^-_\kappa$ follow from $\BA_\lambda^-$? Similarly, can Theorem \ref{medini} be improved to $\BA_\kappa^-$ implies $\mfb > \kappa$?
\end{question}
We note, as discussed in \cite{Stepranswatson87}, that for the full Baumgartner axioms for Polish spaces this is also open. For instance it is not known if $\BA_{\aleph_1} (2^\omega)$ follows from $\BA_{\aleph_2}(2^\omega)$.

Also left open from Section 2 is whether the $\BA^-$ axioms were ever equivalent to the full $\BA$ axioms. A precise question along these lines is the following.
\begin{question}
    Is $\BA_\kappa^- (2^\omega)$ equivalent to $\BA_\kappa (2^\omega)$? What about for $\baire$?
\end{question}

Turning to cardinal characteristics there are many questions left open about the application of Baumgartner type axioms to the study of cardinal characteristics. In fact little is known beyond Theorem \ref{medini} above, and in particular aside from the bounding number we have very little information. A sample question along these lines is the following. 

\begin{question}
    Does $\BA^-_\kappa$ imply $\cov (\Me) \neq \kappa$? Does $\BA$ imply $\cov(\Me) > \aleph_1$?
\end{question}

Moving to the universal axioms we note that we showed many non-implications of the $\mathsf{U}_\kappa$ axioms but few implications. Thus a general question is what do the $\mathsf{U}_\kappa$ axioms actually imply? A test question along these lines is the following.

\begin{question}
    What consequences of $\mathsf{U}_\kappa$ or $\mathsf{U}_{\kappa, \lambda}(X)$ hold? For instance do any of these axioms have an effect on the cardinal $\mfd$?
\end{question}

We note here that in Shelah's model discussed in Theorem \ref{shelah} we have $\non (\Me) < \cov (\Me)$ which is in some sense the defining cardinal characteristic in the Cohen model so $\mathsf{U}_{\aleph_1}$ is independent of this inequality. However I don't know of a model $\mathsf{U}_\kappa$ for some $\kappa > \mfd$ except for the trivial case where $\kappa = \cc$. Similarly one can ask whether $\mathsf{U}_{\aleph_1}$ holds in other well known models given its failure in the Cohen and random models. For instance the following is open.

\begin{question}
    Does $\mathsf{U}_{\aleph_1}$ hold in the Miller, Sacks or Laver models? 
\end{question}

Along the same lines we observed that Miller's theorem from \cite{Miller83} implies that there are no $\cc$-dense $X, Y \subseteq 2^\omega$ so that $X$ strongly avoids $Y$. However I do not know whether there are $\aleph_1$-dense such in this model. 

\begin{question}
    Are there $\aleph_1$-dense $X, Y \subseteq 2^\omega$ which strongly avoid each other in the Miller, Sacks or Laver models? 
\end{question}

Since the Miller and Laver models are also models of $\cov(\Null) = \cov(\Me) = \aleph_1$ the following is also open.
\begin{question}
    Are there $\cc$-dense $X, Y \subseteq 2^\omega$ with $X$ strongly avoids $Y$ in the Laver or Miller model?
\end{question}

We also still do not fully understand the difference between strongly avoiding and totally avoiding. As such we ask the following.

\begin{question}
    Is it consistent that there are $X, Y \subseteq 2^\omega$ which are $\cc$-dense and strongly avoid one another but no such which totally avoid one another?
\end{question}

As a test question to unravel the above we can also look at the following. 
\begin{question}
    How are avoiding, strongly avoiding and totally avoiding affected by forcing axioms like $\MA$ and $\PFA$?
\end{question}

We also hope for further applications of Theorem \ref{iteration}. The issue is that we do not know in general which forcing notions preserve avoiding/strongly avoiding/totally avoiding.
\begin{question}
    Suppose $X$ avoids/strongly avoids/totally avoids $Y$ for some spaces $X$ and $Y$. For which families of forcing notions can we prove preserve these properties of $X$ and $Y$? For instance, do Souslin ccc forcing notions, or even simple examples of such like Hechler forcing?
\end{question}

Finally, in this paper we have avoided discussing the full Baumgartner axiom for Polish spaces even though it is the initial motivation. We finish this paper with a conjecture about these axioms. On the face of it the conjecture is somewhat ambitious, yet we make it in part to highlight how little is known.  
\begin{conjecture}
    Let $X$ be a perfect Polish space. Exactly one of the following holds.
    \begin{enumerate}
        \item $X$ contains a closed nowhere dense $F \subseteq X$ so that any homeomorphism $h:X \to X$ maps $F$ to itself and hence $\BA (X)$ provably fails.

        \item Part 1 fails and $X$ is not one dimensional. In this case $\BA (X)$ is equivalent to $\mfp > \aleph_1$. 

        \item Part 1 fails and $X$ is one dimensional. In this case $\BA (X)$ is equivalent to $\BA$. Moreover if $X$ is like this and $Y$ is as in Part 2 then $\BA (X)$ implies $\BA (Y)$. 
    \end{enumerate}
\end{conjecture}

Note that the condition described in Part 1 of the conjecture is a simple topological obstruction to a space satisfying $\BA_\kappa (X)$ since any dense set intersecting the closed nowhere dense set cannot be mapped to one not intersecting it. Examples include finite dimensional manifolds with boundary (where the boundary is the nowhere dense set). Part of the conjecture therefore is that this simple obstruction in fact characterizes the provable failure of a Baumgartner like axiom.

\bibliographystyle{plain}
\bibliography{weakbaumgartner}
\end{document}